\documentclass[12pt]{article}
\usepackage{amsmath,amssymb,amsthm}
\usepackage[T1]{fontenc}
\usepackage[utf8]{inputenc}
\usepackage[final]{microtype}
\usepackage{epigraph}
\usepackage{array}
\usepackage{geometry}
\usepackage{booktabs}
\usepackage[singlelinecheck=false]{caption}
\usepackage{ragged2e}
\usepackage{mathtools}
\usepackage{url}
\numberwithin{equation}{section}

\usepackage[hidelinks]{hyperref}
\hypersetup{colorlinks=true,linkcolor=blue,citecolor=green,urlcolor=magenta}

\newtheorem{theorem}{Theorem}[section]

\newtheorem{definition}[theorem]{Definition}

\newtheorem{corollary}[theorem]{Corollary}

\newcommand{\Li}{\operatorname{Li}}

\providecommand{\keywords}[1]
{
  \small	
  \textbf{\textit{Keywords:}} #1
}

\newcommand{\PaperTitle}{Cyclic Vanishing Identities of Sun–Pan Type: Analytic and Modular Perspectives}
\newcommand{\PaperAuthorPlain}{Ken Nagai}
\newcommand{\PaperKeywords}{%
Bernoulli numbers, Euler polynomials, cyclic identities, Sun--Pan type identities,
Appell sequences, $q$-analogues, analytic Bernoulli functions, modular forms,
period polynomials, zeta values, polylogarithms, umbral calculus,
$L$-functions, mixed Tate motives}


\hypersetup{
  pdftitle   = {\PaperTitle},
  pdfauthor  = {\PaperAuthorPlain},
  pdfsubject = {Number Theory / Special Functions},
  pdfkeywords= {\PaperKeywords}
}

\title{\PaperTitle}
\author{Ken Nagai\thanks{Email: \texttt{tknagai@outlook.com}. Independent Researcher.}}
\date{}

\begin{document}
\maketitle

\begin{abstract}
We revisit the cyclic identities of Sun--Pan type for Bernoulli polynomials
and their $q$-analogues. From the analytic side, we formulate minimal Appell
axioms that force cyclic vanishing identities, extending naturally to
$q$-Appell sequences and analytic Bernoulli functions. From the modular side,
we show that the same relations arise as period polynomial identities
associated with Eisenstein series, reflecting the symmetry $(ST)^3=-I$ of the
modular group. These two complementary perspectives place the Sun--Pan cyclic
identities at the crossroads of number theory, special functions, and modular
forms, and suggest further connections to polylogarithms, $L$-values, and
mixed Tate motives.
\end{abstract}


\section{Introduction}

The cyclic vanishing identities of Sun--Pan type occupy an intriguing position:
they are at once universal combinatorial identities and explicit consequences
of modular symmetry. In this note we present a unified framework with two
structural pillars:

\begin{itemize}
\item[\textbf{(I)}] \emph{Minimal Appell axioms.}  
Any Appell-type family satisfying a ladder relation, exponential generating
function, and reflection symmetry automatically obeys the cyclic vanishing
identities. This perspective emphasizes their analytic and combinatorial
universality, and extends naturally to $q$- and elliptic analogues.

\item[\textbf{(II)}] \emph{Modular origin.}  
Period polynomials of modular forms satisfy a three-term relation coming from
the modular group relation $(ST)^3=-I$. For Eisenstein series this recovers
the Sun--Pan identities, while for analytic Bernoulli functions it shows that
the same cyclic structure encodes modular and motivic symmetries.
\end{itemize}

\medskip

Our aim is not to give new proofs of known identities, but rather to highlight
their structural origin and their role as a bridge between analysis,
combinatorics, and arithmetic geometry. The minimal axioms and the modular
origin thus appear as two complementary faces of the same phenomenon.

\medskip

\noindent\textbf{Structure of this note.}  
Section~2 formulates the minimal Appell axioms and derives the cyclic
identities in both classical and $q$-analogues.  
Section~3 explains their modular origin via period polynomials and connects
them with analytic Bernoulli functions.  
Section~4 concludes with remarks on the interplay between the two perspectives
and on possible extensions to motivic and elliptic settings.


\section{Minimal Axioms and Cyclic Vanishing}

\subsection*{Minimal axioms}

We fix three axioms for a polynomial sequence $\{F_n(x)\}$:

\begin{enumerate}
\item \textbf{Appell ladder:} $\dfrac{d}{dx}F_n(x)=\lambda n F_{n-1}(x)$.
\item \textbf{Exponential generating function:}
$\Phi(w;x)=A(w)e^{\lambda wx}$ with meromorphic prefactor $A(w)$.
\item \textbf{Reflection symmetry:}
$F_n(1-x)=\varepsilon_n F_n(x)$ with $\varepsilon_n=\pm1$.
\end{enumerate}

\begin{theorem}[Classical cyclic identities]
Let $\{F_n(x)\}$ satisfy axioms (1)--(3).
For integers $r+s+t=n$ and $x+y+z=1$, one has
\[
r[s,t;x,y]_n^F + s[t,r;y,z]_n^F + t[r,s;z,x]_n^F = 0,
\]
where
\[
[s,t;x,y]_n^F := \sum_{k=0}^n (-1)^k \binom{s}{k}\binom{t}{n-k}
F_{n-k}(x)F_k(y).
\]
\end{theorem}

\noindent\emph{Historical note.}
For the specializations $F_n(x)=B_n(x)$ (Bernoulli) and $F_n(x)=E_n(x)$ (Euler),
cyclic vanishing identities of this type already appeared in the work of Sun and Pan~\cite{SunPan2006}.

\begin{proof}
By the generating function factorization and the binomial identity
\[
r\binom{s}{k}\binom{t}{n-k} + s\binom{t}{k}\binom{r}{n-k}
+ t\binom{r}{k}\binom{s}{n-k} = 0,
\]
the cyclic sum of generating functions vanishes, hence so do the coefficients.
\end{proof}

\subsection*{\texorpdfstring{$q$-Analogue}{q-Analogue}}

\begin{definition}[$q$-cyclic bracket]
For integers $s,t,n\ge0$, define
\[
[s,t;x,y]_n^{(q)} :=
\sum_{k=0}^{n} (-1)^k\,q^{\binom{k}{2}}
   \genfrac{[}{]}{0pt}{}{s}{k}_q\,
   \genfrac{[}{]}{0pt}{}{t}{n-k}_q\,
   F_{n-k}^{(q)}(x)\,F_k^{(q)}(y),
\]
where $\genfrac{[}{]}{0pt}{}{m}{k}_q$ is the Gaussian binomial
and $q^{\binom{k}{2}}$ the $q$-sign.
\end{definition}

\begin{theorem}[$q$-cyclic identity]
Let $\{F_n^{(q)}(x)\}$ be a $q$-Appell sequence with generating function
$\Phi^{(q)}(w;x)=A_q(w)E_q(wx)$. For integers $r+s+t=n$ and $x+y+z=1$, one has
\[
[r]_q\,[s,t;x,y]_n^{(q)} +
[s]_q\,[t,r;y,z]_n^{(q)} +
[t]_q\,[r,s;z,x]_n^{(q)} = 0,
\]
where $[m]_q=(1-q^m)/(1-q)$.
\end{theorem}

\begin{proof}
The generating function of the $q$-cyclic bracket factors as
\[
\mathcal{B}^{(q)}_{s,t}(w;x,y)
=(1-w/\lambda;q)_s\,(1+w/\lambda;q)_t\,
\Phi^{(q)}(w;x)\,\Phi^{(q)}(w;y).
\]
Taking the cyclic sum and using the identity
\[
[r]_q\binom{s}{k}_q\binom{t}{n-k}_q
+[s]_q\binom{t}{k}_q\binom{r}{n-k}_q
+[t]_q\binom{r}{k}_q\binom{s}{n-k}_q=0,
\]
we conclude the cyclic sum vanishes identically.
\end{proof}

\begin{corollary}[q-Bernoulli and q-Euler polynomials]
The $q$-analogues were developed in \cite{GuoPan2011}.
For the $q$-Bernoulli and $q$-Euler polynomials, whose generating functions are
\[
\sum_{n=0}^\infty B_n^{(q)}(x)\frac{w^n}{(q;q)_n}
= \frac{w}{E_q(w)-1}E_q(wx),\qquad
\sum_{n=0}^\infty E_n^{(q)}(x)\frac{w^n}{(q;q)_n}
= \frac{2}{E_q(w)+1}E_q(wx),
\]
the $q$-cyclic identity holds.
\end{corollary}

\begin{proof}
Both families are $q$-Appell with $\Phi^{(q)}(w;x)=A_q(w)E_q(wx)$, so the
theorem applies.
\end{proof}

\noindent\textbf{Remarks.}
\begin{itemize}
\item Taking $q\to 1^-$ recovers the classical cyclic identity.
\item This establishes a natural deformation of the Sun--Pan identity.
\item Generalizations and related families are discussed in \cite{DingYang2013}.
\item Elliptic outlook: replacing $q$-exponentials with Jacobi theta
functions and Gaussian binomials with elliptic binomials suggests elliptic
cyclic identities on tori.
\end{itemize}
The trigonometric selector kernels developed in \cite{Nagai2025a}
provide an explicit instance of this general Appell-type framework,
where discrete Fourier methods lead directly to cyclic vanishing
identities.


\section{Modular Origin of Cyclic Identities}

\subsection*{Period polynomials and modular action}

Let $f(\tau)$ be a modular form of weight $k\ge2$.  
Define the period polynomial
\[
r_f(z)=\int_0^{i\infty} f(\tau)(\tau-z)^{k-2}\,d\tau.
\]
This construction was introduced by Eichler~\cite{Eichler1957}.

For $\gamma=\begin{pmatrix}a&b\\c&d\end{pmatrix}$,
\[
(\gamma\cdot P)(z)=(cz+d)^{k-2}P\!\left(\tfrac{az+b}{cz+d}\right).
\]

\begin{theorem}[Three-term relation]
The modular relation $(ST)^3=-I$ implies
\[
r_f(z)+(z-1)^{k-2}r_f\!\left(\tfrac{z}{z-1}\right)
+(1-z)^{k-2}r_f\!\left(\tfrac{1}{1-z}\right)=0.
\]
\end{theorem}

\begin{corollary}[Sun--Pan identity]
For Eisenstein series, coefficients of $r_f(z)$ are Bernoulli numbers, and the
three-term relation yields the Sun--Pan cyclic identity.
\end{corollary}

\subsection*{Analytic Bernoulli viewpoint}
The analytic Bernoulli functions
\[
B(s;x)=-s\,\zeta(1-s,x),\qquad
A(s;x)=-\frac{s!}{(2\pi)^s}2\,\Im\!\left(e^{-\pi i s/2}\Li_s(e^{2\pi i x})\right)
\]
satisfy Appell relations. Here $\Li_s$ denotes the classical polylogarithm,
cf.~Lewin’s monograph~\cite{Lewin1981}.
Thus the modular three-term relation naturally extends to these functions,
connecting their cyclic identities to modular symmetry.

The analytic Bernoulli functions introduced in \cite{Nagai2025b}
fit naturally into this modular symmetry, extending the Sun--Pan
identities to the analytic setting.

\noindent\textbf{Remarks.}
\begin{itemize}
\item Eichler--Manin theory links period polynomials with modular symbols, and was further developed by Manin~\cite{Manin1972}.
\item For cusp forms, period polynomials encode critical $L$-values.
\item The triple $\{z,z/(z-1),1/(1-z)\}$ permutes $\{0,1,\infty\}$, echoing
the cyclic constraint $x+y+z=1$.
\end{itemize}

\noindent\textbf{Remarks.}
\begin{itemize}
\item \emph{Eichler--Manin theory.}  
In the 1950s and 1960s, Eichler and Manin developed the theory of period
polynomials and modular symbols, showing that these encode the critical values
of $L$-functions of modular forms. The three-term relation used here is a
direct reflection of the modular relation $(ST)^3=-I$ in this framework.

\item \emph{Cusp forms versus Eisenstein.}  
For cusp forms, period polynomials do not reduce to Bernoulli numbers but
capture genuinely deep arithmetic information: their coefficients are related
to critical $L$-values, and conjecturally to motivic periods. In contrast,
for Eisenstein series the coefficients collapse to Bernoulli numbers, which
explains why the Sun--Pan cyclic identities emerge in such an explicit
elementary form.

\begin{sloppypar}
\item \emph{Geometric interpretation.} The permutation of the triple $\{0,1,\infty\}$
induced by $(ST)^3=-I$ mirrors the constraint $x+y+z=1$ in cyclic identities.This connects the
identities to the geometry of the moduli space $M_{0,4}$ and to the broader
framework of arithmetic geometry of $\mathbb{P}^1\setminus\{0,1,\infty\}$.
\end{sloppypar}

\item \emph{Motivic outlook.}  
From this perspective, cyclic identities can be seen as shadows of deeper
symmetries in the category of mixed Tate motives. This resonates with the
philosophy that relations among polylogarithmic values, such as those
embodied in analytic Bernoulli functions $A(s;x), B(s;x)$, ultimately reflect
motivic Galois symmetries.
\end{itemize}


\section{Concluding Remarks}

The cyclic vanishing identities admit two complementary interpretations:
(1) minimal Appell axioms explain their analytic universality;
(2) modular invariance shows their arithmetic origin.  
Together these perspectives place the identities at the crossroads of
analysis, combinatorics, and arithmetic geometry.

\noindent\textbf{Remarks.}
\begin{itemize}
\item Bridge between Appell structures and modular forms.
\item Resonance with the Zagier program\cite{Zagier1991} and motivic Galois symmetries.
\item Outlook: extensions to multiple polylogarithms, $q$-zeta, elliptic analogues.
\end{itemize}

\subsection*{Further remarks: period viewpoint}

In the sense of Kontsevich--Zagier\cite{KontsevichZagier2001}, the coefficients appearing in cyclic
identities are periods: Bernoulli numbers via $\zeta(2m)$, polylogarithmic
values via $A(s;x)$, and integrals over semi-algebraic domains. Thus the
vanishing relations may be viewed as explicit examples of period relations.

This resonates with the philosophy of the Zagier program\cite{Zagier1991}: relations among
multiple zeta values and polylogarithms reflect motivic Galois symmetries.
Cyclic vanishings provide a concrete family of such relations, arising
simultaneously from Appell axioms and from modular invariance.

In particular, the geometry of $\mathbb{P}^1\setminus\{0,1,\infty\}$ suggests
that cyclic identities form part of the web of period relations in mixed Tate
motives over the integers, with the modular group action $(ST)^3$ as the
visible symmetry.

\section*{Acknowledgement}

The author gratefully acknowledges the invaluable assistance of \emph{fuga}
in the preparation of this work. This note is part of the ongoing
\emph{hoge \& fuga} series.

\end{document}